\documentclass[12pt,reqno]{amsart}

\usepackage{hyperref}  

\usepackage{graphicx}
\usepackage{amsfonts}
\usepackage{amssymb}
\usepackage{amsmath}
\usepackage{amsthm}

\newtheorem{theorem}{Theorem}[section]

\newtheorem{lemma}[theorem]{Lemma}

\newcommand{\del}{\backslash}
\newcommand{\con}{/}

\newcommand{\cT}{\mathcal{T}}

\newcommand{\cS}{\mathcal S}

\newcommand{\CR}{\mbox{CR}}

\begin{document}

\sloppy

\title{Taking minors without splitting tangles}

\author{Jim Geelen}
\address{Department of Combinatorics and Optimization,
University of Waterloo, Waterloo, Canada}

\begin{abstract} We prove that
any element in a matroid can be removed, by either deletion or contraction, in such a way that no tangle  ``splits".
\end{abstract}


\maketitle

\section{Introduction}
 
 Let $N$ be a minor of a matroid $M$ and let $\cT$ be a tangle of order $k$ in $M$. We say that $\cT$ {\em splits} in $N$ if there are two distinct tangles of order $k$ in $N$
 that both induce the tangle $\cT$ in $M$  (see Section~\ref{sec2} for the definition of a tangle and for the definition of ``induce"). We prove the following result.
 \begin{theorem}\label{main}
 Any element in a matroid can be removed by either deletion or contraction in such a way that no tangle  splits.
 \end{theorem}
 We also prove a version of this result for pivot-minors of graphs; see Theorem~\ref{pm}.
 
\section{Connectivity systems and tangles}\label{sec2}

A {\em connectivity system} is a pair $(E,\lambda)$ where $\lambda$ is a non-negative integer-valued function defined on the subsets of a finite set $E$ such that
\begin{itemize}
\item for each partition $(X,Y)$ of $E$ we have $\lambda(X)=\lambda(Y)$, and
\item for sets $X,Y\subseteq E$ we have
$ \lambda(X\cap Y)+\lambda(X\cup Y)\le \lambda(X)+\lambda(Y).$
\end{itemize}
For a set $X\subseteq E$, we refer to $\lambda(X)$ as the {\em connectivity} of $X$.
For a connectivity system $K=(E,\lambda)$ we let
$\cS_k(K)$ denote the collection of all sets $X\subseteq E$  of connectivity less than $k$.

In this section we review tangles; we start by recalling the definition from~\cite{GGRW}. 
A {\em tangle} of order $k$ in a connectivity system $K=(E,\lambda)$ is a set $\cT\subseteq \cS_k(K)$ such that
\begin{itemize}
\item for each set $A\in \cS_k(K)$, exactly one of $A$ and $E\setminus A$ is contained in $\cT$,
\item there are no three sets in $\cT$ whose union is $E$, and
\item no set in  $\cT$  has size  $|E|-1$.
\end{itemize}

Let $K=(E,\lambda)$ and $K_0=(E_0,\lambda_0)$ be connectivity systems. We say that $K$ {\em dominates} $K_0$ if
$E_0\subseteq E$ and $\lambda_0(X)\le \lambda(X)$ for all $X\subseteq E_0$.
If $K$ dominates $K_0$ and $\cT_0$ is a tangle of order $k$ in  $K_0$, then taking $\cT$ to be the collection of those sets 
$X\in \cS_k(K)$ such that $S\cap E(N)\in \cT_0$ clearly defines a tangle of order $k$ in $K$. We say that $\cT$ is the tangle {\em induced} by $\cT_0$.
Then we say that a tangle $\cT$ of order $k$ in $K$ {\em splits} in $K_0$, if there are two distinct tangles $\cT_1$ and $\cT_2$ of order $k$ in $K_0$ that both induce the tangle $\cT$  in $K$.

We say that $K_0$ {\em adheres} to $K$ if $K$ dominates $K_0$ and, for 
each partition $(A,B)$ of $E(N)$, there is a partition $(X_1,X_2)$ of either $A$ or $B$ such that $\lambda(X_1)\le \lambda_0(A)$
and $\lambda(X_2)\le\lambda_0(A)$. We note that we are allowing one of $X_1$ and $X_2$ to be empty here.
 
 \begin{lemma}\label{unsplit}
If $K_0$ is a connectivity system that adheres to a connectivity system $K$, then no tangle in $K$ splits in $K_0$.
\end{lemma}

\begin{proof} Let $K=(E,\lambda)$ and $K_0=(E_0,\lambda_0)$.
Suppose that there is a tangle $\cT$ of order $k$ in $K$ that splits into two distinct tangles $\cT_1$ and $\cT_2$ of order $k$ in $K_0$.
There is a partition $(A_1,A_2)$ of $E_0$ such that $A_1\in \cT_1$ and $A_2\in \cT_2$; let $t=\lambda_0(A_1)$.
Up to symmetry we may assume that there is a partition $(X_1,X_2)$ of $A_1$ with $\lambda(X_1)\le t$ and $\lambda(X_2)\le t$.

Since $A_1\in \cT_1$ and $K_0$ cannot be covered by two sets in $\cT_1$, we have that
$\cT_1$ contains both $X_1$ and $X_2$. On the other hand, since $A_2\in \cT_2$ and $K_0$ cannot be covered by three sets in $\cT_2$, we have that
$\cT_2$ contains at least  one of  $X_1$ and $X_2$; up to symmetry we may assume that $X_2\in \cT_2$.
Since $\lambda(X_2)\le t$ and since $\cT$ is induced by both $\cT_1$ and $\cT_2$, we have $X
_2\in \cT$ and $E\setminus X_2\in \cT_2$, contrary to the tangle axioms.
\end{proof}

\section{Matroids}

For a set $X$ of elements in a matroid $M$ we define
$$ \lambda_M(X) := r(X)+r(E\setminus X) - r(M) +1.$$
The connectivity system associated with $M$ is $K(M)=(E(M),\lambda_M)$.

We rely on the following result; see Oxley~\cite[Lemma 8.5.3]{Oxley}.
\begin{lemma}\label{BC-ineq}
Let $A$ and $B$ be sets of elements in a matroid $M$. If $e$ is an element disjoint from $A$ and $B$, then
$$ \lambda_{M\del e}(A) + \lambda_{M\con e}(B) \ge \lambda_M(A\cap B) +\lambda_M(A\cup B\cup\{e\}) -1.$$
\end{lemma}

This lemma implies that an element can be either deleted or contracted from a matroid in a way that does
not do too much damage to the connectivity. 
\begin{lemma}\label{bixby}
If $e$ is an element in a matroid $M$, then at least one of $K(M\del e)$ and $K(M\con e)$ adheres to $K(M)$.
\end{lemma}

\begin{proof} If $K(M\del e)$ does not adhere to $K(M)$, then there is a partition $(A_1,A_2)$ of $E(M\del e)$ such that neither
$A_1$ nor $A_2$ admits a partition into two set of connectivity at most $\lambda_{M\del e}(A_1)$ in $M$.
If $K(M\con e)$ does not adhere to $K(M)$, then there is a partition $(A_3,A_4)$ of $E(M\con e)$ such that neither
$A_3$ nor $A_4$ admits a partition into two set of connectivity at most $\lambda_{M\con e}(A_3)$ in $M$.
Let $k_1=\lambda_{M\del e}(A_1)$ and $k_2=\lambda_{M\con e}(A_3)$.

Up to symmetry and duality we may assume that $\lambda(A_2\cap A_4)$  is the largest of $\lambda(A_1\cap A_3)$, $\lambda(A_1\cap A_4)$, $\lambda(A_2\cap A_3)$, and $\lambda(A_2\cap A_4)$. By our choice of $(A_1,A_2)$ and $(A_3,A_4)$ we have
$\lambda(A_2\cap A_4)>\max(k_1,k_2)$. Then, by Lemma~\ref{BC-ineq}, we have $\lambda(A_1\cap A_3)\le \min(k_1,k_2)$. Then, again 
by our choice of $(A_1,A_2)$ and $(A_3,A_4)$, we have
$\lambda(A_1\cap A_4)>k_1$ and $\lambda(A_2\cap A_3)>k_2$. But then we get a contradiction by applying  Lemma~\ref{BC-ineq} to $A_1$ and $A_4$.
\end{proof}

Note that Theorem~\ref{main} is implied by Lemmas~\ref{unsplit} and~\ref{bixby}.

\section{Pivot minors}

Let $A$ be the adjacency matrix of a simple graph $G=(V,E)$. The {\em cut-rank} of a set $X\subseteq V(G)$, denoted  $\rho_G(X)$,  is defined
to be the rank of the submatrix $A[X,V\setminus X]$, over GF$(2)$. Oum~\cite{Oum} shows that $\CR(G):=(V,\rho_G)$
is a connectivity system.

We assume familiarity with pivot minors; see~\cite{Oum}.
For an edge $e=uv$, we let $G\times e$ denote the graph obtained by pivoting on the edge $e$. The connectivity system $\CR(G)$ is invariant under pivoting; that is $\CR(G)=\CR(G\times e)$. We recall that there are effectively only two ways to remove a vertex under pivot minors; we can either delete a vertex or pivot on an edge incident with that vertex and then delete the vertex. More specifically, if $e$ is any edge incident with a vertex $v$ and
$H$ is a pivot minor of $G$ not containing $v$, then $H$ is a pivot minor of either $G-v$ or of $(G\times e)-v$.

The following result was proved by Oum~\cite{Oum}.
\begin{lemma}\label{pm-ineq}
Let $A$ and $B$ be vertex sets in a simple graph $G$. If $v$ is a vertex disjoint from $A$ and $B$ and $e$ is any edge incident with $v$, then
$$ \rho_{G-v}(A) + \rho_{(G\times e)-v}(B) \ge \rho_G(A\cap B) +\rho_G(A\cup B\cup\{e\}) -1.$$
\end{lemma}

The proof of the following lemma is essentially the same as that of Lemma~\ref{bixby}.
\begin{lemma}\label{pm-bixby}
If $v$ is a vertex of a simple graph $G$ and $e$ is any edge incident with $v$, then at least one of $\CR(G-v)$ and $\CR((G\times e)-v)$ adheres to $\CR(G)$.
\end{lemma}

\begin{proof} If $\CR(G-v)$ does not adhere to $\CR(G)$, then there is a partition $(A_1,A_2)$ of $V(G)\setminus\{v\}$ such that neither
$A_1$ nor $A_2$ admits a partition into two set of cut-rank at most $\rho_{G-v}(A_1)$ in $G$.
If $\CR((G\times e)-v)$ does not adhere to $\CR(G)$, then there is a partition $(A_3,A_4)$ of $V(G)\setminus\{v\}$ such that neither
$A_3$ nor $A_4$ admits a partition into two set of cut-rank at most $\rho_{(G\times e)-v}(A_3)$ in $G$.
Let $k_1=\rho_{G-v}(A_1)$ and $k_2=\rho_{(G\times e)-v}(A_3)$.

Up to symmetry and pivoting we may assume that $\rho_G(A_2\cap A_4)$  is the largest of $\rho_G(A_1\cap A_3)$, $\rho_G(A_1\cap A_4)$, $\rho_G(A_2\cap A_3)$, and $\rho_G(A_2\cap A_4)$. By our choice of $(A_1,A_2)$ and $(A_3,A_4)$ we have
$\rho_G(A_2\cap A_4)>\max(k_1,k_2)$. Then, by Lemma~\ref{pm-ineq}, we have $\rho_G(A_1\cap A_3)\le \min(k_1,k_2)$. Then, again 
by our choice of $(A_1,A_2)$ and $(A_3,A_4)$, we have
$\rho_G(A_1\cap A_4)>k_1$ and $\rho_G(A_2\cap A_3)>k_2$. But then we get a contradiction by applying  Lemma~\ref{pm-ineq} to $A_1$ and $A_4$.
\end{proof}

The following result is implied by Lemmas~\ref{unsplit} and~\ref{pm-bixby}.
\begin{theorem}\label{pm}
If $v$ is a vertex of a simple graph $G$ and $e$ is any edge incident with $v$, then there is at least one choice of $H$ in $\{G-v,\,(G\times e)-v\}$ 
such that no tangle of $\CR(G)$ splits in $\CR(H)$.
\end{theorem}

For a vertex $v$ of $G$ we let $G *v$ denote the graph obtained from $G$ by locally complementing at $v$. Readers who are familiar with vertex minors will see that Theorem~\ref{pm} implies:
\begin{theorem}\label{mm}
If $v$ is a vertex of a simple graph $G$ and $e$ is any edge incident with $v$, then for at least two choices of $H$ among $(G-v,\,(G\times e)-v,\, (G*v)-v)$ no tangle of $\CR(G)$ splits in $\CR(H)$.
\end{theorem}

\section{Entangled connectivity systems}

We  call a connectivity system {\em $k$-entangled} if for each $t\le k$ there is at most one tangle of order $t$. We call a matroid {\em $k$-entangled} when its connectivity system is. Since $k$-connected matroids are $k$-entangled, we consider $k$-entanglement to be a weakening of $k$-connectivity. Finding elements to delete or contract keeping $k$-connectivity is difficult or impossible, even for relatively small $k$; see, for example,~\cite{CMO,GW,Hall}. However, as an immediate application of Theorem~\ref{main} we get:
 \begin{theorem}\label{entangled}
 Any element in a $k$-entangled matroid can be removed by either deletion or contraction keeping it $k$-entangled.
 \end{theorem}

We conclude with some observations on the structure of $k$-entangled connectivity systems which should help in applications of Theorem~\ref{entangled}.
Note that, if a connectivity system $K=(E,\lambda)$ has no tangle of order $k$ then it is $k$-entangled if and only if it is $(k-1)$-entangled. Thus we may as well assume that $K$ 
has a tangle $\cT$ of order $k$. We will explain structurally, without explicit reference to tangles, what distinguishes the sets in $\cT$ from the other sets in $\cS_k(K)$.

A tree is {\em cubic} if each of its vertices has degree $1$ or $3$; the degree-$1$ vertices are the {\em leaves}. A {\em partial branch-decomposition} of $K$ is a pair $(T, f)$
where $T$ is a cubic tree and $f$ is a function that maps the elements of $E$ to the leaves of $T$. 
A leaf $v$ of $T$ is said to {\em display} a set $X\subseteq E$ if $X$ is the set of all elements that $f$ maps to $v$. For an edge $e$ of $T$, the graph $T-e$ has two components and and we partition $E$ into two sets $(A_e,B_e)$ according to which component they are mapped to by $f$. The {\em width} of the edge $e$ is $\lambda(A_e)$ and the {\em width} of the partial tree-decomposition is the maximum of its edge-widths. 

\begin{lemma}\label{easy1}
Let $\cT$ be a tangle of order $k$ in a  connectivity system $K=(E,\lambda)$. If $(T,f)$ is a partial branch-decomposition of width at most $k-1$ then there is a unique leaf of $T$ that displays a set that is not in $\cT$.
\end{lemma}

\begin{proof} There is at most one leaf that displays a set not in $\cT$ since otherwise we can cover $K$ with two sets in $\cT$.

Suppose that each leaf displays a set in $\cT$.
Each edge $e$ determines a partition $(A_e,B_e)$ of $E$ and $\cT$ contains exactly one of $A_e$ and $B_e$; we orient the edge away from the side containing the set in $\cT$. Note that, in particular,  the edges are oriented away from the leaves. This orientation has no directed cycles, so there is a vertex with all incident edges oriented towards it.
But then we can cover $E$ with three sets in $\cT$, contrary to the definition of a tangle.
\end{proof}

A set $X\subseteq E$ is {\em $k$-branched} if there is a partial branch decomposition $(T, f)$ of $K$ of width at most $k$ such that $E\setminus X$ is displayed by a leaf and every other leaf displays at most one element of $X$. We say that $X$ is {\em weakly $k$-branched} if $\lambda(X)\le k$ and there is a partial branch decomposition $(T, f)$ of $K$ of width at most $k$ such that each leaf either displays a subset of $E\setminus X$ or it displays a singleton. It is an easy consequence of Lemma~\ref{easy1} that every  weakly $(k-1)$-branched set is contained in every tangle of order at least $k$. The converse also holds for $k$-entangled connectivity systems.

\begin{lemma}\label{branched}
Let $\cT$ be a tangle of order $k$ in a $k$-entangled connectivity system $K=(E,\lambda)$. Then a set $X\in \cS_k(K)$ is weakly $\lambda(X)$-branched if and only if $X\in\cT$.
\end{lemma}

The harder direction of Lemma~\ref{branched} is implied by the following result that is proved in~\cite[Theorem~3.3]{GGRW}.
For a collection $\cS$ of subsets of $E$, we say that a partial branch-decomposition $(T,f)$ {\em conforms} with $\cS$ if each of its leaves displays a subset of a set in $\cS$.
\begin{lemma} \label{duality}
Let $K=(E,\lambda)$ be a connectivity system and let $\cS\subseteq \cS_k(K)$ such that the union of the sets in $\cS$ is $E$.
Then $\cS$ extends to a tangle of order $k$ if and only if there is no partial branch-decomposition of width at most $k-1$ that conforms to $\cS$.
\end{lemma}
To prove Lemma~\ref{branched}, consider a set $X\in \cT$. Take the partition $\cS$ of $E$ consisting of $E\setminus X$ and all singleton subsets of $X$. Since $K$ is $k$-entangled,
we cannot extend $\cS$ to a tangle of order $\lambda(X)+1$, so, by Lemma~\ref{duality}, there is a partial branch-decomposition of width at most $\lambda(X)$ that conforms to $\cS$.
Therefore $X$ is weakly $\lambda(X)$-displayed.

Lemma~\ref{branched} can be refined, but we need the following result. For disjoint sets $X$ and $Y$ in $K$ we define $\kappa_K(X,Y)$ to be the minimum of $\lambda(Z)$ taken over all
sets $Z$ with $X\subseteq Z\subseteq E(M)\setminus Y$.
\begin{lemma}\label{linked}
Let $(T, f)$ be a partial branch-decomposition in a connectivity system $K=(E,\lambda)$.
Now let $Y$ be the set displayed by a leaf $r$, let $X\subseteq E\setminus Y$, and let  $(T, f')$ be the partial branch decomposition obtained by defining $f'(e)=f(e)$ for all $e\in X$ and 
defining $f'(e)=r$ for all other elements $e\in E\setminus X$. If $\lambda(X) = \kappa_K(X,Y)$,
then the weight of each edge in $(T, f')$ is at most the weight of the same edge in $(T,f)$.
\end{lemma}

\begin{proof}
Let $e$ be an edge of $T$ and let $(A,B)$ be the partition of $E$ such that $B$ is the set of elements mapped by $f$ to the component of $T-e$ containing $r$.
The the weight of $e$ in $(T,f)$ is $\lambda(A)$ and the weight of $e$ in $(T,f')$ is $\lambda(A\cap X)$. Since $\lambda(X) = \kappa_K(X,Y)$, we have $\lambda(X\cup A)\le \lambda(X)$. Then
$$
\lambda(A\cap X) \le \lambda(A)+\lambda(X) -\lambda(X\cup A) 
\le \lambda(A)+\lambda(X) -\lambda(X\cup A) \le \lambda(A),
$$
as required.
\end{proof}

A set in $\cT$ is called {\em $\cT$-linked} if there is no set $Y\in \cT$ with $X\subseteq Y$ and $\lambda(Y)<\lambda(X)$. 
\begin{lemma}\label{branched2}
If $\cT$ is a tangle of order $k$ in a $k$-entangled connectivity system $K=(E,\lambda)$, then every $\cT$-linked set $X\in \cT$ is $\lambda(X)$-branched.
\end{lemma}

\begin{proof} Let $t=\lambda(X)$.
By Lemma~\ref{branched}, the set $X$ is weakly $t$-branched, so there is a
is a partial branch decomposition $(T, f)$ of $K$ of width at most $t$ such that each leaf either displays a subset of $E\setminus X$ or it displays a singleton.
By Lemma~\ref{easy1}, there is a leaf, say $r$, that displays a set $Y$ that is not in $\cT$. Since $E\setminus Y\in \cT$, the set $Y$ is not a singleton and hence $Y$ is disjoint from $X$.
Since $X$ is $\cT$-linked, we have $\kappa_K(X,Y)=\lambda(X)$. But then, by Lemma~\ref{linked}, we can re-map all elements in $E\setminus X$ to $r$, and hence $X$ is $t$-branched.
\end{proof}

We conclude by interpreting the above results in the context of matroids for some small values of $k$.
Let $M$ be a $k$-entangled matroid and let $\cT$ be a tangle of order $k$ in $K(M)$.

For $k=2$, every set in $\cT$ is $1$-branched, and such sets consists of loops and co-loops of $M$. So $M$ has exactly one component with two or more elements.

Now consider $k=3$ and suppose that $M$ is connected. Every set in $\cT$ is $2$-branched and
any $2$-branched set admits a series-parallel reduction to a single element. Therefore $M$ admits a series-parallel reduction to a $3$-connected matroid with at least $4$-elements.

Finally consider $k=4$ and suppose that $M$ is $3$-connected. Then every set in $\cT$ is $3$-branched. Here the structure becomes a bit more interesting, but simple enough to be potentially useful.
For example, if the matroid is binary, then we can interpret a partial branch decomposition of width $3$ as being constructed via a sequence of $3$-sums.

\section{Acknowledgment}

I thank Geoff Whittle for suggesting the extension of Theorem~\ref{entangled} to Theorem~\ref{main}.


\begin{thebibliography}{}

\bibitem{CMO}
C. Chun, D. Mayhew, J. Oxley,
A chain theorem for internally 4-connected binary
matroids
J. Combin. Theory, Ser. B, 101 (2011), 141-189.

\bibitem{GGRW}
J. Geelen, B. Gerards, N. Robertson, G. Whittle, 
Obstructions to branch-decomposition of matroids, 
J. Combin. Theory, Ser. B 96 (2006), 560-570.

\bibitem{GW}
J. Geelen, G. Whittle,
Matroid 4-Connectivity: A Deletion–Contraction Theorem,
J. Combin. Theory, Ser. B, 83 (2001), 15-37.

\bibitem{Hall}
R. Hall, A chain theorem for 4-connected matroids, J. Combin. Theory Ser. B 93 (2005) 79-100.

\bibitem{Oum}
S. Oum, Rank-width and vertex-minors, J. Combin. Theory Ser. B 95 (2005) 45–66.

\bibitem{Oxley}
J. Oxley, \textit{Matroid theory, second ed.}, Oxford University Press, New York, (2011).

\end{thebibliography}
\end{document}